\documentclass[runningheads]{llncs}
\pdfoutput=1

\usepackage{amssymb}  	
\usepackage{mathrsfs}	
\usepackage[all]{xy}	

\usepackage{graphicx}

\usepackage[ruled,vlined,linesnumbered]{algorithm2e}
\usepackage{empheq}
\usepackage{blkarray,bigstrut}
\usepackage{tikz}
\usepackage{algpseudocode}

\usepackage{cleveref}
\crefname{algocf}{alg.}{algs.}
\Crefname{algocf}{Algorithm}{Algorithms}

\newcommand{\ie}{\mbox{\it i.e.}}
\newcommand{\etal}{\mbox{\it et al.}}

\newcommand{\void}[1]{}				

\newcommand{\rr}{\mathbb R}

\newcommand{\zz}{\mathbb Z}

\newcommand{\Lap}[1]{\mathcal{L}}

\newcommand{\thisway}{$(\Longrightarrow)$}
\newcommand{\thatway}{$(\Longleftarrow)$}

\renewenvironment{proof}
{\noindent{\it Proof:}\hspace{.5em}\rm}
{\hfill$\square$

\vspace{0.1em}}

\newenvironment{proofof}[1]			
	{\begin{trivlist}\item[]{\it Proof of #1:\hspace{.5em}}\rm}
	{\hfill$\square$
	\end{trivlist}}
	{\begin{trivlist}\item[]{\it Proof sketch:\hspace{.5em}}\rm}
	{\hfill$\square$
	\end{trivlist}}

\numberwithin{equation}{section}

\newcommand{\U}{{\mathcal{U}}}
\renewcommand{\L}{{\mathcal{L}}}
\newcommand{\upperbd}[2]{\beta^+(\U_{#1, #2})}
\newcommand{\lowerbd}[2]{\beta^-(\L_{#1, #2})}
\newcommand{\upperbdpath}[2]{\mathcal{U}_{#1,#2}}
\newcommand{\lowerbdpath}[2]{\mathcal{L}_{#1,#2}}

\newcommand{\bm}{\boldsymbol}

\newcommand{\dist}[2]{\Pi(#2)-\Pi(#1)}
\newcommand{\nedge}[0]{null-edge}

\newcommand{\diagincmat}{Robinson matrix}

\newcommand{\path}[1]{\langle#1\rangle}	
\newcommand{\bma}{\bm a}
\newcommand{\bmb}{\bm b}
\newcommand{\bmc}{\bm c}
\newcommand{\bmd}{\bm d}

\newcommand{\calS}{\mathcal S}

\newcommand{\edge}[2]{\{#1, #2\}}
\newcommand{\reverse}[1]{#1^{\leftarrow}}
\newcommand{\rev}{^{\leftarrow}}
\newcommand{\bmD}{\mathbb D^k}

\begin{document}
\title{Uniform Embeddings for Robinson Similarity Matrices\thanks{Supported by an NSERC Discovery grant.}}
\titlerunning{Uniform embeddings for Robinson matrices}
%
\author{Jeannette Janssen\inst{1} \and
Zhiyuan Zhang\inst{1}}
\authorrunning{J. Janssen and Z. Zhang }
%
\institute{Dalhousie University, Halifax, NS, Canada.\\
\email{\{Jeannette.Janssen, owen.zhang\}@dal.ca} }
\maketitle              
\begin{abstract}
A Robinson similarity matrix is a symmetric matrix where all entries in all rows and columns are increasing towards the diagonal.
A Robinson matrix can be decomposed into the weighted sum of $k$ adjacency matrices of a nested family of unit interval graphs. We study the problem of finding an embedding which gives a \emph{simultaneous} unit interval embedding for all graphs in the family. This is called a \emph{uniform embedding.}
We give a necessary and sufficient condition for the existence of a uniform embedding, derived from paths in an associated graph. We also give an efficient combinatorial algorithm to find a uniform embedding or give proof that it does not exist, for the case where $k=2$.   

\end{abstract}

\section{Introduction}

In many different settings it occurs that a linearly ordered set of data items is given, together with a pair-wise similarity measure of these items, with the property that items are more similar if they are closer together in the ordering. A classic example of this setting is in archaeology, where sites are ordered according to their age, and the composition of the items found at the sites are more similar if the sites are closer in age. Other applications occur in evolutionary biology, sociology, text mining, and visualization. (See \cite{liiv2010seriation} for an overview.) The similarity between such an ordered set of items, when presented in the form of a matrix, will have the property that entries in each row and column increase towards the diagonal (when items are closer in the ordering), and decrease away from the diagonal. Such a matrix is called a  Robinson matrix, or Robinson similarity matrix. 

Formally, a {\em Robinson matrix} is a symmetric matrix where the entries $a_{i,j}$ satisfy the following condition:
\begin{equation}\label{robinson:intro}
\text{For all $u<v<w$, } a_{u,v}\geq a_{u,w}\text{ and }a_{v,w}\geq a_{u,w}.
\end{equation}
In other words, a Robinson matrix is an asymmetric matrix where entries in each row and column are increasing towards the diagonal. See Example 1 (to follow) for examples.
Robinson matrices are named after Robinson, who first mentioned such matrices in \cite{robinson} in the context of archaeology. 

Robinson wanted to solve the following question, that is also referred as the {\em seriation problem}: suppose a set of objects has an underlying linear order and given their pair-wise similarity, arrange the objects so that objects that are closer in the arrangement are more similar than pairs that are further apart. That is,
 find the ordering of the items for which the similarity values form a Robinson matrix. However, seriation only gives a {\sl linear ordering} of the items. In this paper we focus on finding a linear representation of the items that takes into account the numerical value of the similarity. In the context of archaeology, this would mean that we are looking not only the order of the sites in terms of their age, but also of some indication of their age. 

We assume similarities cannot be judged very precisely, and thus we focus on Robinson matrices where the entries are taken from a restricted set $
\{ 0,1,2\dots, k\}$, where $k$ indicates ``very similar", and 0 indicates ``not at all similar". We will assume throughout that all diagonal entries equal $k$. We are then looking for a linear embedding of the items so that the distance between pairs with the same similarity value are approximately similar.  More precisely, we require that there exist threshold distances $d_1>d_2>\dots >d_k>0$ and an embedding of the items into $\rr$, such that, if  a pair of items has similarity level $t$, then the distance between their embedded values lies between threshold distances $d_{t+1}$ and $d_t$. We will call this a \emph{uniform embedding}; See Definition \ref{def:uniform embedding} for a formal definition.

A $\{ 0,1\}$-valued symmetric matrix $A$ is Robinson if and only if $A-I$ is the adjacency matrix of a \emph{proper interval graph} \cite{kendall1969incidence}. 
The class of proper interval graphs equals the class of unit interval graphs
 \cite{roberts1969indifference}, which equals the class of \emph{indifference graphs}. A graph is an indifference graph if and only if there exists a linear embedding of the vertices with respect to a threshold distance $d>0$ so that two vertices are adjacent if and only if their embedded values have distance at most $d$. An indifference graph embedding is therefore a uniform embedding for the associated binary Robinson matrix.

A Robinson matrix taking values in $\{ 0,1,\dots ,k\}$ can be seen as the representation of a nested family of indifference graphs. Namely, any such matrix $A=(a_{u,v})$ can be written as
$A =  \sum_{t = 1}^k A^{(t)},$ 
where for all $t\in [k]$, $A^{(t)}=(a^{(t)}_{u,v})$ is a binary matrix such that, $a^{(t)}_{u,v}=1$ if $a_{u,v}\geq t$, and 0 otherwise. Clearly, each $A^{(t)}$ is Robinson and has all ones on the diagonal. Therefore $A^{(t)}-I$ is the adjacency matrix of an indifference graph $G^{(t)}$. These graphs are nested, i.e. for all $t<k$, $G^{(t+1)}$ is a subgraph of $G^{(t)}$. In this light, our problem can be restated as that of finding a \emph{simultaneous} indifference graph embedding for all graphs $G^{(t)}$.

As shown in the following example, not every Robinson matrix has a uniform embedding. \begin{example}\label{example:counter-example}
Consider the following matrices 
$$
A = \left[\begin{array}{ccccc}
 2& 2& 1& 0& 0 \\ 
 2& 2& 2& 1& 1 \\
 1& 2& 2& 2& 1 \\
0& 1& 2& 2& 2 \\
0& 1& 1& 2& 2 \\
\end{array}\right]
\qquad
B = \left[
\begin{array}{cccccc}
 2& 2& 1& 0& 0& 0 \\
2& 2& 2& 1& 1& 1 \\
1& 2& 2& 2& 1& 1 \\
0& 1& 2& 2& 2& 1 \\
0& 1& 1& 2& 2& 2 \\
0& 1& 1& 1& 2& 2 \\
\end{array}
\right]
$$ 
Matrix $A$ has a linear embedding $\Pi$ with threshold distances $d_1=8, d_2=6$ given by
\[
\Pi = \path{0, 5, 6.5, 11.75, 12.75}.
\]
We can check that, for any pair $(i,j)$, if  $a_{i,j}=2$ then $\Pi(i)$, $\Pi(j)$ have distance at most 6, if $a_{i,j}=1$ then the distance between $\Pi(i)$ and $\Pi(j)$ lies between 6 and 8, and if  $a_{i,j}=0$ then this distance is greater than 8.

In contrast, matrix $B$ does not have a uniform embedding. Suppose there exists such an embedding $\Pi$ with threshold distances $d_1> d_2>0$. Suppose also that $\Pi$ is increasing; see Theorem 1 for the justification. 
Since $b_{1,4}=0$ and $b_{4,6}=1$, we have that $\dist 14> d_1$ and $\dist 46>d_2$. This implies that  $\dist16 > d_1 + d_2$. 
On the other hand, we have that $b_{1,2}=2$ and $b_{2,6}=1$, so
$\dist12\leq d_2$ and  $\dist26\leq d_1$. This implies that $\dist16\leq d_1+d_2$. 
Combining the inequalities results in  $d_1 + d_2<\dist16\leq d_1+d_2$,  a contradiction. 

\end{example}

In this paper, we consider the  problem of finding a uniform embedding of a \diagincmat{} or giving proof that it does not exist. In Theorem \ref{thm:main theorem} we give a condition for the existence of a uniform embedding in terms of the threshold distances $d_1,\dots , d_k$. We then show that this condition is sufficient by giving an algorithm to find a uniform embedding, given threshold distances that meet the condition. We will also compare the complexity of verifying the given condition with the complexity of solving the inequality system defining a uniform embedding (Definition \ref{def:uniform embedding}).

Finally, we consider the case where $k=2$, so the problem is to find a uniform embedding for two nested indifference graphs. We give a combinatorial algorithm to either find a uniform embedding or give a substructure that shows it does not exist. The algorithm has complexity $O(N^{2.5})$, where $N=n(n-1)/2$ is the size of the input.

\subsection{Related Works}
In \cite{chuangpishit2017uniform}, the problem of finding a uniform embedding was studied for \emph{diagonally increasing graphons}. A graphon is a symmetric function $w:[0,1]^2\rightarrow [0,1]$. Graphons can be seen as generalizations of matrices. The results from \cite{chuangpishit2017uniform} do not apply in the context of this paper. Namely, a matrix can be represented as a graphon, but the ``boundaries" delineating the regions of $[0,1]^2$ where the graphon takes a certain value $1\leq t\leq k$ in this case are piecewise constant functions. The results from  \cite{chuangpishit2017uniform} apply only to boundary functions that are continuous and strictly increasing. 

In \cite{roberts1969indifference}, Roberts established the equivalence between 
the classes of unit interval graphs, proper interval class, indifference graphs, and {\em claw-free} interval graphs. Different proofs can be found in  \cite{bogart1998short} and 
\cite{gardi07}.

A lot of work has been done on the seriation problem. For binary matrices, this is equivalent to recognizing proper (or unit) interval graphs. 
Corneil \cite{corneil20043-sweep} gives a linear time unit interval recognition algorithm, which improves on \cite{corneil1995simple}. 
Atkins \etal \cite{atkins1998spectral} gave  a spectral algorithm for the general seriation problem.  
Laurent and Seminaroti  \cite{laurent2017lex,laurent2017similarity}  give a combinatorial algorithm that generalizes the algorithm from  \cite{corneil20043-sweep}. A general overview of the seriation problem and its applications can be found in \cite{liiv2010seriation}.

\section{Uniform Embeddings}\label{chapter: the main result}

We start by formally defining a uniform embedding. First we introduce some notation. For an integer $t\in \zz_+$, let $[t]=\{1,2,\dots ,t\}$. Let $\mathcal{S}^n[k]$ denote the set of all Robinson matrices with entries in $\{ 0,1,2,\dots ,k\}$, and let $\bmD$ be the set of threshold vectors  $\bmD = \{\bmd\in\rr^k:\bmd= (d_i)_{i\in[k]}, d_1>\dots> d_k>0\}.$

\begin{definition}\label{def:uniform embedding}
Given a matrix \emph{$A\in \mathcal{S}^n[k]$ and a threshold vector $\bmd \in\bmD $}, a map $\Pi:[n]\to \rr$ is a \emph{uniform embedding of $A$ with respect to $\bmd$} if, for each pair $u,v\in[n]$:
\begin{equation}\label{eqn:shortcut}
    a_{u,v} = t\iff d_{t+1}<|\Pi(v)-\Pi(u)|\leq d_{t} \quad\text{ for }t\in\{0,\dots, k\},
\end{equation}
where we define $d_{k+1}=-\infty$ and $d_0=\infty$, so that the lower bound for $a_{u,v}=k$ and the upper bound for $a_{u,v}=0$ are trivially satisfied. 
    \end{definition}

The following theorem states that, if a uniform embedding exists, then we may assume it to have certain nice properties. The proof is straightforward but technical, and has been omitted here; it can be found in the Appendix. 

\begin{theorem}\label{thm:monotone}
Let $A\in \mathcal{S}^n[k]$. If $A$ has a uniform embedding, then there exist  $\bmd\in\bmD$ and a uniform embedding $\Pi$ with respect to $\bmd$ which is strictly monotone increasing, and which is such that the inequalities \ref{eqn:shortcut} are all strict. That is,
for all pairs of $u,v\in[n]$, $u<v$,
\begin{equation}\label{eqn:shortcut2}
    a_{u,v} = t\iff d_{t+1} < \dist uv < d_t,
\end{equation}
where $d_{k+1}=0$ and $d_0=\infty$.
\end{theorem}

Note that the definition of $d_{k+1}$ in Theorem \ref{thm:monotone} has changed from $-\infty$ to zero; this change enforces that $\Pi$ is \emph{strictly} increasing. In the later context, we will assume that any uniform embedding $\Pi$ of matrix $A$ with respect to any $\bmd$ is a  map $\Pi:[n]\rightarrow \rr$ which satisfies \cref{eqn:shortcut2} (with the new definition of $d_{k+1}$).  This will simplify the proofs and reduce the need to distinguish different cases. By Theorem \ref{thm:monotone}, we can make this assumption without loss of generality.

 \section{Bounds, Walks, and Their Concatenation}\label{sec:Bounds and walks}

The contradiction for matrix $B$ in the Example \ref{example:counter-example} was derived from a cyclic sequence of vertices, namely $\path{1,4,6, 2,1}$, and the bounds on the distance between successive pairs of this sequence. In this section, we show how walks in an associated graph generate a set of bounds which need to be satisfied by any linear embedding.

\begin{definition}\label{def:bound}
Let $A\in\mathcal{S}^n[k]$ and fix $u,v\in[n]$.
A vector $\bmb\in\zz^k$ is an {\em upper bound on $(u,v)$} if 
the inequality $\dist uv < \bmb^\top\bmd$ is implied by the inequality system (\ref{eqn:shortcut2}) in the sense that, for any uniform embedding $\Pi$ and threshold vector $\bmd\in\bmD$ satisfying inequality system (\ref{eqn:shortcut2}), it holds that $\dist uv< \bmb^\top\bmd$.
Similarly, the vector  $\bmb$ is a {\em lower bound on $(u,v)$} if the inequality $\bmb^\top\bmd<\dist uv$ is implied by (\ref{eqn:shortcut2}). 

\end{definition}

It follows directly from inequality system (\ref{eqn:shortcut2}) that, for any matrix $A\in\mathcal{S}^n[k]$, and any pair $u,v\in[n]$, $ u<v$, the all-zero vector $\bm 0$ is a lower bound on $(u,v)$. 
Note that, if $\bmb$ is an upper bound on $(u,v)$, then $-\bmb$ is a lower bound on $(v,u)$, and vice versa.

We will see how new bounds can be obtained from walks in a corresponding graph. 
An original set of bounds, derived from edges, can be obtained directly from inequality system (\ref{eqn:shortcut2}). This is made precise in the following definition.
Let $\bm\chi_i\in\zz^k$ denote the unit vector with $1$ at the $i$th position and zero otherwise.

\begin{definition}\label{def:bounds}
Let $A=(a_{i,j})\in\mathcal{S}^n[k]$. 
Let $u,v\in [n]$, $u<v$. Define
$$
 \beta^+(u,v)= \left\{\begin{array}{ll}
    \bm\chi_{t} & \text{if }a_{u,v}=t\geq 1; \\  
    \text{undefined } &\text{if } a_{u,v}=0.
\end{array}\right.
$$
and
$$\beta^-(u,v) = 
\left\{\begin{array}{ll}
    \bm\chi_{t+1} & \text{if }a_{u,v}=t\leq k-1; \\  
    \bm0 &\text{if } a_{u,v}=k.
\end{array}\right.
$$
\end{definition}

It follows immediately from \Cref{def:bounds} and inequality system (\ref{eqn:shortcut2}) that, for any $u,v\in [n]$, $u<v$, $\beta^+(u,v)$ is a lower bound on $(u,v)$ (if $a_{u,v}\not=0$), and $\beta^-(u,v)$ is a lower bound on $(u,v)$.

We can consider $\path{u,v}$ as a walk of length 1 in the complete graph 
with vertex set $[n]$. If $u<v$, then $\beta^+(u,v)$ is the upper bound defined by this walk, and $\beta^-(u,v)$ is the lower bound. We now extend this notion to unordered pairs, and, more generally, to longer walks. To accommodate the fact that there is no upper bound on $\Pi(v)-\Pi(u)$ if $a_{u,v}=0$, we distinguish \emph{edges} in this graph which are pairs $\{u,v\}$ such that $a_{u,v}>0$, and \emph{null-edges} which are pairs $\{ u,v\}$ so that $a_{u,v}=0$. In the following we will see that we can combine walks to obtain more bounds. 

\begin{definition}\label{def:bound-walk}
A \emph{$(u,v)$-walk} is a sequence $W=\path{w_0,w_1,\dots ,w_p}$ where $w_i\in [n]$, $0\leq i\leq p$, and $u=w_0$ and $v=w_p$. In other words, $W$ is a walk in the complete graph with vertex set $[n]$. The walk $W$ is an \emph{upper-bound-walk} if for all $1\leq i\leq p$,
$$ \{ w_{i-1},w_i\} \text{ is }\left\{
\begin{array}{ll}
    \text{an edge } & \text{if } w_{i-1}<w_{i},  \\
    \text{an edge or a \nedge{} } & \text{if }w_{i-1}>w_{i}.
\end{array}
\right. 
$$ 
In other words, in an upper-bound-walk, null-edges are only traversed from larger to smaller vertices. 
Similarly, the walk $W$ is a \emph{lower-bound-walk} if null-edges are only traversed to go from smaller to larger vertices. 

Now first define for all $u,v\in [n]$ so that $u<v$,  
\begin{equation}\label{def:beta vu}
\beta^+(v,u)=-\beta^-(u,v)\mbox{ and }\beta^-(v,u)=-\beta^+(u,v).
\end{equation}
Then for any walk $W=\path{w_0, w_1,\dots , w_p}$, define
\begin{eqnarray}
\beta^+ (W)&=\sum_{i=1}^p \beta^+(w_{i-1},w_i)&\mbox{if }W\mbox{ is an upper-bound-walk, and}\label{eqn:beta(W)}\\
\beta^- (W)&=\sum_{i=1}^p \beta^-(w_{i-1},w_i)&\mbox{if }W\mbox{ is a lower-bound-walk.}\label{eqn:betamin(W)}
\end{eqnarray}

Given two walks $W_1 = \path{u_0,\dots, u_s}$ and $W_2 = \path{u_s, \dots, u_p}$, denote $W = W_1 + W_2 = \path{u_0, \dots, u_p}$ as the {\em concatenation} of $W_1$ and $W_2$.
\end{definition}

For any walk $W=\path{w_0,w_1,\dots ,w_{p-1},w_p}$, define the \emph{reverse} of $W$ as $W\rev=\path{w_p,w_{p-1},\dots , w_1, w_0}$. Clearly, if $W$ is an upper-bound-walk, then $W\rev$ is a lower-bound-walk, and vice versa. By (\ref{def:beta vu}), (\ref{eqn:beta(W)}) and (\ref{eqn:betamin(W)}), we have that  $\beta^+(W)=-\beta^-(W\rev)$.

\begin{lemma}\label{lemma: concatenation}
Let $A\in \mathcal{S}^n[k]$.
For any $u,v\in [n]$ and any $(u,v)$-walk $W$, if $W$ is an upper-bound-walk then $\beta^+(W)$ is an upper bound on $(u,v)$, and if $W$ is a $(u,v)$-lower-bound-walk then $\beta^-(W)$ is a lower bound on $(u,v)$.
\end{lemma}

The proof of this lemma follows by induction on the length of the walk, using the definitions. It can be found in the Appendix.

\section{A Sufficient and Necessary Condition}\label{sec:main_thm}

\Cref{sec:Bounds and walks} introduced the necessary concepts to state the main theorem of this paper. We saw in the previous section that upper- and lower-bound-walks give bounds that must be satisfied by any uniform embedding. This hints at a condition for the existence of a uniform embedding: there must exist $\bmd\in\bmD$ so that each lower bound derived from a $(u,v)$-walk is smaller than each upper bound derived from a $(u,v)$-walk. 

As it turns out, we only need to consider lower- and upper-bound-paths, that is, walks without repeated vertices. 
Given a matrix $A\in \mathcal{S}^n[k]$, let $\lowerbdpath uv$ be the set of all $(u,v)$-lower-bound-paths, and $\upperbdpath uv$ be the set of all $(u,v)$-upper-bound-paths. Note that $\upperbdpath uv$ and $\lowerbdpath uv$ are finite, whereas the set of all walks is infinite. 

Define the inequality system:\\
For all $u,v\in [n]$, $u<v$, for any upper bound $\bmb=\beta^+(W_1)$ where $W_1\in \upperbdpath uv$ and any lower bound $\bma=\beta^-(W_2)$ where $W_2\in\lowerbdpath uv$,
\begin{equation}\label{eqn:condition_d}
 \bma^\top\bmd<\bmb^\top\bmd.
\end{equation}

\begin{theorem}\label{thm:main theorem}
A \diagincmat{} $A\in \mathcal{S}^n[k]$ has a uniform embedding if and only if there exists $\bmd\in\bmD$ satisfying inequality system (\ref{eqn:condition_d}).
\end{theorem}

We can prove the necessity of \Cref{thm:main theorem} without other tools

\begin{proofof}{the forward implication of \Cref{thm:main theorem}}
Suppose $A$ has a uniform embedding. By Theorem \ref{thm:monotone}, this implies that $A$ has a uniform embedding $\Pi$ with respect to a  threshold vector $\bmd\in\bmD$ which satisfy inequality system (\ref{eqn:shortcut2}). Let $u,v\in[n]$ with $u<v$. Let $W_1\in\lowerbdpath uv$ and $W_2\in \upperbdpath uv$,  and let $\bma =\beta^-(W_1)$ and $\bmb=\beta^+(W_2)$. Then by Lemma \ref{lemma: concatenation} and Definition \ref{def:bound}, 
$\bma^\top\bmd<\Pi(v)-\Pi(u) < \bmb^\top\bmd.
$
\end{proofof}

For the converse, we will obtain an iterative procedure to construct a uniform embedding $\Pi$ which satisfies inequality system \cref{eqn:shortcut2}. However, first we need to prove that condition (\ref{eqn:condition_d}) for paths implies that the same condition holds for all walks. 

\subsection{Cycles and Paths}\label{sec:partial order}

In Section \ref{sec:Bounds and walks} we saw how walks can be used to generate new inequalities that are implied by the inequality system \ref{eqn:shortcut2}. 
In this section we show that, for the existence of a uniform embedding we need only to consider paths.

A $(u,v)$-upper-bound-walk $W=\path{u=w_0,w_1,\dots ,w_p=v}$ is an {\em upper-bound-cycle} if $u=v$. and $W$ contains no other repeated vertices. Note that the order in which the cycle is traversed determines whether or not it is an upper-bound-cycle.

\begin{lemma}\label{lemma:ub cycle geq 0}
Let $A\in\calS^n[k]$  and $\bmd\in\bmD$.
Let $C=\path{u_1,\dots, u_p}$, $u_1=u_p=u$, be an upper-bound-cycle.
If $\bmd\in\bmD$ satisfies \cref{eqn:condition_d}, then $\beta^+(C)^\top\bmd > 0$.
\end{lemma}
\begin{proof}
Suppose $v=u_i\in C$ for some $1<i<p$, then $C=W_1 + W_2$ where 
$W_1=\path{u_1,\dots, u_i}$  and 
$W_2 = \path{u_i,\dots, u_p}$. Then $W_1$ is a $(u,v)$-upper-bound-path, and $W_2$ is a $(v,u)$-upper-bound-path, so $W_2\rev$ is a $(u,v)$-lower-bound-path.
Then by \Cref{def:bound-walk},
$$
\beta^+(C) = \beta^+(W_1)+\beta^+(W_2)=\beta^+(W_1)-\beta^-(W_2\rev).
$$
By the choice of $\bmd$, 
$\beta^-(W_2^{\rev} )^\top \bmd < \beta^+(W_1 )^\top \bmd$, and thus $\beta^+(C)^\top\bmd > 0$.

\end{proof}

\begin{lemma}\label{lemma:loopless}
Let $A\in\calS^n[k]$  be a \diagincmat{} and let $\bmd\in\bmD$, and suppose $\bmd$ satisfies (\ref{eqn:condition_d}). 
Suppose $W$ is a $(u,v)$-upper-bound-walk $W$. Then there exists a $(u,v)$-upper-bound-path $W'$ such that $\beta^+(W')^\top\bmd
\leq \beta^+(W)^\top \bmd$. If $W$ is a $(u,v)$-lower-bound-walk, then there exists a $(u,v)$-lower-bound-path $W'$ such that $\beta^-(W')^\top\bmd
\geq \beta^-(W)^\top \bmd$.
\end{lemma}

The proof follows easily from the previous lemma and the well-known fact that each walk can be transformed into a path by successively removing cycles. Details of the proof can be found in the Appendix. We now have the following corollary.

\begin{corollary}
Let $A\in\calS^n[k]$  be a \diagincmat{} and let $\bmd\in\bmD$. If $\bmd$ satisfies (\ref{eqn:condition_d}), then for every $u,v\in [n]$, for every $(u,v)$-upper-bound-walk $W_1$ and every $(u,v)$-lower-bound-walk $W_2$,
\begin{equation*}
 \beta^-(W_1)^\top \bmd<\beta^+(W_2)^\top\bmd.
\end{equation*}
\end{corollary}

\subsection{Finding a Uniform Embedding}
In this section we prove the converse of Theorem \ref{thm:main theorem}. That is, given a matrix $A\in \mathcal{S}^n[k]$ we assume that there exists a $\bmd\in\bmD$ satisfying inequality system (\ref{eqn:condition_d}), and we show that there exists a uniform embedding $\Pi$ with respect to this particular threshold vector $\bmd$.
We present an iterative formula to calculate $\Pi:[n]\rightarrow \rr$, given $\bmd$ and the sets $\upperbdpath uv, \lowerbdpath uv$ for all $u,v\in [n]$, $u<v$. For brevity, let $\upperbd uv= \{\beta^+(W):W\in\upperbdpath uv\}$ and  $\lowerbd uv= \{\beta^-(W):W\in\lowerbdpath uv\}$. Define $\Pi$  as follows:
\begin{equation}\label{alg:embedding algorithm}
    \begin{array}{rl}
     \Pi(1)&= 0  \\
     \Pi(v)&= (ub_v+lb_v)/2, \quad\text{for }2\leq v\leq n, 
\end{array}
\end{equation}
where $ub_v, lb_v$ are defined iteratively using $\Pi(1),\dots ,\Pi(v-1)$ as:
\begin{equation}\label{equi:ubv and lbv}\begin{split}
    ub_v & = \min_{i\in[v-1]}\left\{ \Pi(i)+ \min\{\bmb^\top\bmd: \bmb\in\upperbd iv\} \right\},\\ 
    lb_v & = \max_{i\in[v-1]}\left\{ \Pi(i)+ \max\{\bma^\top\bmd: \bma\in\lowerbd iv\} \right\}.
\end{split}\end{equation}
The following two lemmas show that $\Pi$ defined as such is a uniform embedding of $A$ with respect to $\bmd$.

\begin{lemma}\label{lemma:lb<ub}
The map $\Pi$ as defined in (\ref{alg:embedding algorithm}) and (\ref{equi:ubv and lbv}) is strictly increasing.

\end{lemma}
\begin{proof} 
We prove by induction on $v$ that $\Pi$ is increasing on $[v]$. The base case, $v=1$, is trivial. For the induction step, fix $v\geq 2$ and assume $\Pi$ is increasing on $[v-1]$.

Note first that $lb_v\geq \Pi(v-1)$. Namely, $\path{v-1,v}\in \lowerbdpath {v-1}v$, and thus either $\bm{\chi}_t\in \lowerbd {v-1}v$ for some $t\in [k]$, or $\bm{0}\in \lowerbd {v-1}v$. Therefore,
$\lowerbd {v-1}v$ contains at least one lower bound $\bm{a}$ so that  $\bma^\top \bmd \geq 0$.

We now show that $lb_v<ub_v$. This suffices to show that $\Pi$ is strictly increasing: if $lb_v<ub_v$ then 
$$\Pi(v)=(lb_v+ub_v)/2>lb_v\geq \Pi(v-1).
$$
Let $u, w$ be the vertices attaining $ub_v$ and $lb_v$ respectively,
and 
$\bmb_{\min}\in\upperbd uv$ such that 
$\bmb_{\min}^\top\bmd=\min\{\bmb^\top\bmd:\bmb\in\upperbd uv\}$, 
$\bma_{\max}\in\lowerbd wv$ such that 
$\bma_{\max}^\top\bmd=\max\{\bma^\top\bmd:\bma\in\lowerbd wv\}$, \ie,
\begin{equation*}\begin{split}
     ub_v = \Pi(u) + \bmb_{\min}^\top\bmd &= \min_{i\in[v-1]}\left\{ \Pi(i)+ \min\{\bmb^\top\bmd: \bmb\in\upperbd iv\} \right\},   \\
     lb_v = \Pi(w) + \bma_{\max}^\top\bmd & =\max_{j\in[v-1]}\left\{ \Pi(j)+ \max\{\bma^\top\bmd: \bma\in\lowerbd jv\} \right\}.
\end{split}\end{equation*}

Let $W_B$  be a $(u,v)$-upper-bound-path such that $\bmb_{\min} = \beta^+(W_B)$ and $W_A$ a $(w,v)$-lower-bound-path such that $\bma_{\max} = \beta^-(W_A)$. 
    Suppose first that $u=w$.
    Then, $\Pi(u)=\Pi(w)$, and $\bma_{\max}\in\lowerbd uv$ and $\bmb_{\min}\in\upperbd uv$. By the choice of $\bmd$, $\bma_{\max}^\top \bmd<\bmb_{\min}^\top \bmd$, and thus $lb_v < ub_v$.
    
    Suppose next that $u\not= w$. Then the concatenation $W_B + W_A\rev{}$ is a $(u,w)$-upper-bound-walk and $\beta^+( W_B + \reverse{W_A}) = \bmb_{\min} - \bma_{\max}$. By Lemma \ref{lemma: concatenation} this implies that  $\dist uw < (\bmb_{\min}-\bma_{\max})^\top\bmd$. This results in
    $$lb_v = \Pi(w)+\bma_{\max}^\top\bmd < \Pi(u) + \bmb_{\min}^\top\bmd = ub_v.$$

\end{proof}

\begin{lemma}\label{theorem:correctness of Pi}
Given \diagincmat{} $A\in\calS^n[k]$, and let $\Pi$ be defined as in (\ref{alg:embedding algorithm}). Then $\Pi$ satisfies inequality system (\ref{eqn:shortcut2}).
\end{lemma}

\begin{proof}
Let $u,v\in[n]$ with $u<v$, and let $a_{u,v}=t$.
We need to show that 
$$d_{t+1} < \dist uv < d_t,
$$
where $d_{k+1}=0$ and $d_0=\infty$.

We first prove the upper bound. If $t=0$, then the inequality $\Pi(v)-\Pi(u)<d_0=\infty$ is trivially satisfied.
Suppose then that $t\neq 0$ (so $uv$ is an edge).  Then $\path{u,v}$ is a $(u,v)$-upper-bound-path, so $\beta^+(u,v)\in \upperbd uv$. By Definition \ref{def:bound-walk}, $\beta^+(u,v)^\top \bmd = d_{t}$. 
By Eq.~( \ref{alg:embedding algorithm}), 
\begin{eqnarray*}
\Pi(v)< ub_v &= & \min_{i\in[v-1]}\left\{ \Pi(i)+ \min\{\bmb^\top\bmd: \bmb\in\upperbd iv\} \right\},\\ 
&\leq &  \Pi(u)+ \min\{\bmb^\top\bmd: \bmb\in\upperbd uv\} \\
&\leq &  \Pi (u) + \beta^+(u,v)^\top\bmd =\Pi(u)+d_t.
\end{eqnarray*}

Next we prove the lower bound.
If $t=k$, then the inequality $\Pi(v)-\Pi(u)>d_{k+1}=0$ is satisfied since $\Pi$ is strictly increasing. 
If $0\leq t<k$, then $\path{u,v}$ is a $(u,v)$-lower-bound-path. So  $\beta^-(u,v)\in \lowerbd uv$, and 
\begin{eqnarray*}
\Pi(v)> lb_v &= & \max_{i\in[v-1]}\left\{ \Pi(i)+ \max\{\bmb^\top\bmd: \bmb\in\lowerbd iv\} \right\},\\ 
&\geq &  \Pi(u)+ \max\{\bmb^\top\bmd: \bmb\in\lowerbd uv\} \\
&\geq &  \Pi (u) + \beta^-(u,v)^\top\bmd =\Pi(u)+d_{t+1}.
\end{eqnarray*}

\end{proof}

Thus we have established that $\Pi$ as defined in (\ref{alg:embedding algorithm}) is a uniform embedding.

\section{Testing the Condition}

According to Theoremr \ref{thm:monotone}, a uniform embedding exists if and only if the inequality system (\ref{eqn:shortcut2}) has a solution. The existence of a uniform embedding can therefore be tested, and an embedding found, by using an linear program solver to determine feasibility of the system and, if feasible, find values for the variables $d_i$, $i\in [k]$ and $\Pi (u)$, $u\in [n]$. The condition for the existence of a uniform embedding as expressed in Theorem \ref{thm:main theorem} involves solving another inequality system, namely (\ref{eqn:condition_d}). This system only contains the variables $d_i$, $i\in [k]$, but the number of inequalities equals the number of pairs of lower- and upper-bound-paths. 

In this section, we will first give a bound on the number of inequalities, and compare the size of the two inequality systems. We then give an algorithm for generating all bounds that lead to inequalities for system (\ref{eqn:condition_d}), and discuss its complexity. Finally, we discuss the case where $k=2$, and give a combinatorial algorithm to find a uniform embedding for a given matrix in $\mathcal{S}^n[k]$, or give proof that it does not exist.

\subsection{A Partial Order on Bounds}

Here we define a partial order on bounds and find out we only need the minimal/maximal elements of this partial order for inequality system (\ref{eqn:condition_d}). This will allow us to bound the size of this system.

To bound the number of equalities in (\ref{eqn:condition_d}), first note that any
such inequality  involves a $(u,v)$-upper-bound-path $W_1$ and a $(u,v)$-lower-bound-path $W_2$. Then $W_1+W_2\rev$ is an upper-bound-cycle $C$, and the inequality can be rewritten as
$\beta^+(C)>0$. Thus we can rewrite (\ref{eqn:condition_d}). Let $\mathcal{C}$ be the set of upper-bound-cycles. Then $\bmd$ satisfies (\ref{eqn:condition_d}) if and only if, 
\begin{equation}\label{eqn:cycle_system}
\mbox{For all }C\in\mathcal{C},\quad \beta^+(C)^\top\bmd>0.
\end{equation}
Any cycle in the complete graph with vertex set $[n]$ has length at most $n$. Any edge in the cycle can contribute at most one to the sum of the coefficients of the bound. Thus we have that, for any cycle $C$.
$$
\beta^+(C)\in \zz^k_n:=\{ \bma\in\zz^k: \sum_{i=1}^k |a_i|\leq n\}.
$$
In particular, no coefficient of a path bound can have absolute value more than $n$. This implies that the number upper-bound-cycles, and thus the number of inequalities in (\ref{eqn:cycle_system}) is at most $(2n)^k$. 
Thus, inequality system (\ref{eqn:cycle_system}) has size $O(kn^k)$, while inequality system (\ref{eqn:shortcut2}) has size $O(n^3)$.

However, using the partial order defined below we can give a tighter bound on the number of inequalities. 

\begin{definition}\label{def:boundleq}
Define the relation $\preceq$ on $\zz^k$, such that given any $\bm a=(a_i), \bm b=(b_i)\in\zz^k$,
$$\bm a\preceq \bm b \quad\text{ if }\quad 
\sum_{i=1}^t a_i\leq \sum_{i=1}^tb_i \qquad\text{ for all } t\in[k].$$
\end{definition}

\begin{theorem}\label{thm:bound_equivalence}
Let $\bma, \bmb\in\zz^k$, then
$\bma\preceq \bmb \iff \bma^\top\bmd\leq \bmb^\top\bmd\text{ for all }\bmd\in\bmD.$
\end{theorem}

The proof can be found in the Appendix.

This theorem implies that $\preceq$ is indeed a partial order. More importantly, we have the following corollary.

\begin{corollary}
Fix $u,v\in [n]$. If the inequalities of system (\ref{eqn:cycle_system}) hold for all minimal elements of $\{ \beta^+(C): C\in\mathcal{C}\}$  under $\preceq$, then all inequalities of the system hold. 
\end{corollary}

This implies that the number of inequalities in (\ref{eqn:cycle_system}) is bounded by the number of minimal elements (under $\preceq$), of $\zz^k_n$. The following lemma bounds this set for the special case where $k=2$.

\begin{lemma}\label{lemma:k2minimal}
If $k=2$, then inequality system (\ref{eqn:cycle_system}), including only minimal bounds, has size at most $2n$.
\end{lemma}
\begin{proof}
As argued above, the number of inequalities in this system is bounded by the number of minimal elements of $\zz^2_n$. We will bound this number by giving a decomposition of $\zz^2_n $ into $2n$ chains. The result then follows from Dilworth's theorem, and the fact that all minimal elements form an antichain. 

Fix $t\in [n]$. Let $S_t$ be the set of vectors $\bma \in\zz^2$ so that $|a_1|+\dots +|a_k|=t$. Consider the sets 
\begin{eqnarray*}
S_t^1&=&\{(-t+i,i)^\top:0\leq i\leq t\}\cup \{(i,t-i)^\top:1\leq i< t\},\mbox{ and}\\
S_t^2&=&\{(-t+i,-i)^\top:1\leq i\leq t\}\cup \{(i,-t+i)^\top:1\leq i\leq t\}.
\end{eqnarray*}
Both $S_t^1$ and $S_t^2$ are chains under $\preceq$, and they form a partition of $S_t$. Since $\zz_n^2=\cup_{t=1}^n S_t$, the result follows.
\end{proof}

\subsection{Generating the Bounds}

We employ a variation on the Floyd-Warshall algorithm \cite{fw1962fwalg} to enumerate all upper-bound-paths. See Algorithm \ref{alg:FW-mod} for the pseudocode. This also generates all lower-bound-paths, by reversal.

\SetKwFunction{FWvar}{FW-var}
\begin{algorithm}
  \SetKwData{Left}{left}
  \SetKwFunction{Concat}{Concatenate}
  \SetKwFunction{Path}{Path}
  \SetKwFunction{ubs}{UBW}
  \SetKwFunction{lbs}{LBW}
  \SetKwFunction{minubs}{Minimal-Upper-Bound-Walks}
  \SetKwFunction{maxlbs}{Maximal-Lower-Bound-Walks}
  \SetKwFunction{fwinit}{FW-Initialization}
  \SetKwInOut{Input}{input}\SetKwInOut{Output}{output}

  \Input{A \diagincmat{} $A\in\mathcal{S}^k$}
  \Output{Lookup table $\ubs, \lbs$ defined on $i,j\in[n]$: where 
  $\begin{array}{l}
  \ubs(i,j) = \mbox{all minimal elements of }\upperbdpath ij,\\ 
  \end{array}$
  }
  \BlankLine
  \For{$i\in[n]$}{
    \For{$j= i,\dots, n$}{
      \lIf{$a_{i,j} \neq 0$}{
        $\ubs(i,j)\leftarrow\{\path{i, j}\}$
      }
      $\ubs(j,i)\leftarrow\{\path{j, i}\}$\;
    }
  }
\For{$s=1,\dots, n$}{
\For{$i=1,\dots, n$}{
\For{$j=i,\dots, n$}{
    \ForEach{$W_1\in\ubs(i,s)\text{ and }W_2\in \ubs(s,j)$}{ 
        \If{$W_1+W_2$ is minimal in $\ubs(i,j)\cup\{W_1 + W_2\}$}{
            $\ubs(i,j)\leftarrow \ubs(i,j) \cup\{W_1+W_2\}$\;
            $\ubs(j,i)\leftarrow \ubs(j,i)\cup \{ (W_1+W_2)\rev\}$\;
        }
    }
}}}

  \Return{\ubs}\;  
  \caption{Bound-Generation}\label{alg:FW-mod}
\end{algorithm}

The complexity of this algorithm is dominated by the step where bounds are merged, and thus determined by the size of the set of bounds. If the minimality test in line 9 is implemented by looping through all elements of $S$, then the complexity of the bound-generation algorithm is $O(n^3M^3)$, where $M$ is the number of minimal elements in $\zz^k_n$. 

The minimal upper-bound-paths give insight into the structure of the matrix that constrains the uniform embedding. If inequality system (\ref{eqn:cycle_system}) does not have a solution, then any LP-solver will return a set of $k$ inequalities which, taken together, show the impossibility of fulfilling all constraints. Each of these inequalities is derived from a cycle, and this collection of cycles can be interpreted as the bottleneck that prevents the existence of a uniform embedding. 

\subsection{A Combinatorial Algorithm for the Case $k=2$}

For $k=2$, we can convert the bound generation, testing of condition (\ref{eqn:shortcut2}), and construction of the uniform embedding into a combinatorial algorithm.

Consider inequality system (\ref{eqn:cycle_system}) when $k=2$. Each inequality is of the form  
$a_1d_1+a_2d_2>0$, where $(a_1,a_2)^\top=\beta^+(C)$ for some upper-bound-cycle $C$.
Depending on the sign of $a_2$, $-a_1/a_2$ will give either a lower bound (if $a_2>0$) or an upper bound (if $a_2<0$) on $d_2/d_1$. Thus we can find, in time linear in the number of minimal bounds, the largest lower bound and smallest upper bound on $d_2/d_1$. The inequality system has a solution if and only if the largest lower bound is smaller than the smallest upper bound. 

If the bounds are incompatible and the system has no solution, then the two cycles giving the largest lower bound and smallest upper bound identify those entries  of the matrix that cause the non-existence of a uniform embedding.

Combining the methods we have developed, we now give the steps of the algorithm solving the uniform embedding problem

\subsubsection{Uniform Embedding Algorithm}

Given a matrix $A\in\mathcal{S}^n[k]$, perform the following steps:
\begin{enumerate}
    \item Generate all minimal upper bounds using the Bound-Generation algorithm (\Cref{alg:FW-mod}). Use $\ubs(v,v)$ to extract all minimal upper-bound-cycles.
    \item For each minimal upper-bound cycle, find its associated bound $\beta^+(C)$, and convert the inequality $\beta^+(C)$ into a lower or upper bound on $d_2/d_1$. Only keep the cycles $C_1$ and $C_2$ generating the largest upper bound and the smallest upper bound encountered in each step.
    \item If the largest lower bound is greater than or equal to the smallest upperbound on $d_2/d_1$, then print NO SOLUTION.  {\bf Exit} and return $C_1$ and $C_2$.
    \item If the largest lower bound is smaller than the smallest upper bound on $d_2/d_1$, then choose $d_1,d_2$ so that $d_2/d_1$ lies between these bounds.
    \item Compute a uniform embedding $\Pi$ with respect to $(d_1,d_2)$ using the formula given in (\ref{alg:embedding algorithm}) and (\ref{equi:ubv and lbv}). {\bf Exit} and return $\Pi$.
\end{enumerate}

The complexity of the algorithm is determined by the generation of bounds in the first step. We can use the partition of $\zz^2_n$ into chains $S^1_t$, $S^2_t$, $1\leq t\leq n$, as given in Lemma \ref{lemma:k2minimal}. Given a bound $(a_1,a_2)^\top$, we can identify which set $S^i_t$ the bound belongs to: $t=|a_1|+|a_2|$ and $i=1$  if $a_2>0$ and $i=2$ otherwise. Therefore, in line 9 of the Bound-Generation algorithm we only need to compare $W_1+W_2$ with the unique minimal element of the set $S^i_t$ it belongs to. This can be done in $O(1)$ time.

Since there are at most $2n$ minimal elements in $\ubs(i,j)$, the loop starting in line 8 of the Bound-Generation algorithm takes $O(n^2)$ steps. Therefore, the  Bound-Generation algorithm can be implemented to take $O(n^5)$ steps. Moreover, the algorithm can be easily modified to compute the upper-bound-paths as well as their associated bounds.   
The generation of the cycle inequalities in Step 2 is immediate from $\ubs$, since every upper-bound-cycle will be included in $\ubs(i,i)$, $i\in [n]$. Note that there will be duplications, since each cycle $C$ will be included in $\ubs (v,v)$ for any vertex $v\in C$. Therefore, generating the inequalities and the associated bounds on $d_2/d_1$ takes $O(n^2)$ steps.

If a threshold vector $(d_1,d_2)^\top$ can be found, then the uniform embedding can be found in $O(n^3)$ steps: there are $n$ iterations, and each iteration involves computing $lb_v$ and $ub_v$, which involves looping over all vertices $u<v$ and all bounds in $\ubs (u,v)$. 

The resulting complexity is somewhat higher than that of solving inequality system (\ref{eqn:shortcut2}) with a state-of-the-art Linear Programming solver. Namely, using the method of Vaidya \cite{Vaidya89} , an LP with $n$ variables and $m$ inequalities can be solved in $O((n+m)^{1.5} n)$ time. In (\ref{eqn:shortcut2}) there are $O(n^2)$ inequalities, so the LP solver has time complexity $O(n^{4})$.

Finally, note that the input size of the problem is $N=n(n-1)/2$, namely the number of upper diagonal entries of the matrix. Therefore, the complexity of the algorithm to find a uniform embedding for the case $k=2$ has complexity $O(N^{2.5})$.

\section{Conclusions}

We gave a sufficient and necessary condition for the existence of a uniform embedding of a Robinson matrix, in the form of a system of inequalities constraining the threshold values $d_1,\dots ,d_k$. For Robinson matrices taking values in $\{ 0,1,2\}$, we gave a $O(N^{2.5})$ algorithm which returns a uniform embedding, or returns two cycles that identify the matrix entries that cause a contradiction in the inequalities defining a uniform embedding. 

For Robinson matrices having more than three values, the condition for the existence of a uniform embedding involves solving an inequality system. An interesting question is whether the condition can be tested with a combinatorial algorithm, as in the case $k=2$. In particular, we saw that the problem can also be formulated as that of finding simultaneous embeddings for a family of nested proper interval graphs. The $k=2$ case shows that this is possible for a family of two. For $k=3$, we can find conditions on $d_1,d_2,d_3$ so that any pair of graphs in the family has a simultaneous embedding. If these conditions are contradictory, then no uniform embedding of the family exists. But if they are not, can we then solve the uniform embedding question combinatorially?

\bibliographystyle{splncs04}
\bibliography{ref}

\begin{thebibliography}{10}
\providecommand{\url}[1]{\texttt{#1}}
\providecommand{\urlprefix}{URL }
\providecommand{\doi}[1]{https://doi.org/#1}

\bibitem{atkins1998spectral}
Atkins, J.E., Boman, E.G., Hendrickson, B.: A spectral algorithm for seriation
  and the consecutive ones problem. SIAM Journal on Computing  \textbf{28}(1),
  297--310 (1998)

\bibitem{bogart1998short}
Bogart, K.P., West, D.B.: A short proof that `proper = unit'. Discrete
  Mathematics  \textbf{201}(1),  21--23 (1999)

\bibitem{chuangpishit2017uniform}
Chuangpishit, H., Ghandehari, M., Janssen, J.: Uniform linear embeddings of
  graphons. European Journal of Combinatorics  \textbf{61},  47--68 (2017)

\bibitem{corneil20043-sweep}
Corneil, D.G.: A simple 3-sweep {LBF} algorithm for the recognition of unit
  interval graphs. Discrete Applied Mathematics  \textbf{138}(3),  371--379
  (2004)

\bibitem{corneil1995simple}
Corneil, D.G., Kim, H., Natarajan, S., Olariu, S., Sprague, A.P.: Simple linear
  time recognition of unit interval graphs. Information processing letters
  \textbf{55}(2),  99--104 (1995)

\bibitem{fw1962fwalg}
Floyd, R.W.: Algorithm 97: Shortest path. Commun. ACM  \textbf{5}(6), ~345 (Jun
  1962)

\bibitem{gardi07}
Gardi, F.: The {R}oberts characterization of proper and unit interval graphs.
  Discrete Mathematics  \textbf{307}(22),  2906 -- 2908 (2007)

\bibitem{kendall1969incidence}
Kendall, D.: Incidence matrices, interval graphs and seriation in archeology.
  Pacific Journal of mathematics  \textbf{28}(3),  565--570 (1969)

\bibitem{laurent2017lex}
Laurent, M., Seminaroti, M.: A {Lex-BF}-based recognition algorithm for
  {R}obinsonian matrices. Discrete Applied Mathematics  \textbf{222},  151 --
  165 (2017)

\bibitem{laurent2017similarity}
Laurent, M., Seminaroti, M.: Similarity-{F}irst {S}earch: a new algorithm with
  application to {R}obinsonian matrix recognition. SIAM Journal on Discrete
  Mathematics  \textbf{31}(3),  1765--1800 (2017)

\bibitem{liiv2010seriation}
Liiv, I.: Seriation and matrix reordering methods: An historical overview.
  Statistical Analysis and Data Mining: The ASA Data Science Journal
  \textbf{3}(2),  70--91 (2010)

\bibitem{roberts1969indifference}
Roberts, F.S.: Indifference graphs, pp. 139--146. Academic Press (1969)

\bibitem{robinson}
Robinson, W.S.: A method for chronologically ordering archaeological deposits.
  American Antiquity  \textbf{16}(4),  293--301 (1951)

\bibitem{Vaidya89}
{Vaidya}, P.M.: Speeding-up linear programming using fast matrix
  multiplication. In: 30th Annual Symposium on Foundations of Computer Science.
  pp. 332--337 (1989)

\end{thebibliography}

\clearpage
\appendix
\section{Proof of Theorems}
\subsection{Strict Monotonicity of Uniform Embedding}\label{appendix: proof monotone}
In this section, we give a proof of Theorem \ref{thm:monotone}.
We break down the theorem into several steps.
We first show that the uniform embedding of a Robinson matrix with no repeating rows is always strictly monotone in \Cref{lemma: no repeating}.
Then we show in \Cref{lemma:strictly within} that, if a strictly monotone uniform embedding satisfies Condition (\ref{eqn:shortcut}), then we can construct another uniform embedding with Condition (\ref{eqn:shortcut2}).
By combining \Cref{lemma:strictly within} and \Cref{lemma: no repeating}, we prove \Cref{thm:monotone} by showing that we can place the embedded values of all repeating rows close enough, while the embedding remains strictly increasing.

\begin{lemma}\label{lemma: no repeating}
Let $A\in\calS^n[k]$ be a Robinson matrix with no repeating rows. Suppose $\Pi$ is a uniform embedding of $A$ with respect to $\bmd\in\bmD$, then $\Pi$ is strictly monotone. If $\Pi(1)<\Pi(2)$, then $\Pi$ is strictly increasing; and if $\Pi(1)>\Pi(2)$, then $\Pi$ is strictly decreasing. 
\end{lemma}

\begin{proof}
Let $A\in\calS^n[k]$ be a matrix that contains no repeating rows. 

To prove this lemma, we break it down to two parts. 
First we prove that all vertices are embedded to distinct values.
Second, we prove that the embedding $\Pi$ is monotone. 
And thus, $\Pi$ is strictly monotone.

To prove the first part, assume, toward a contradiction, that $\Pi(u)=\Pi(v)$ for some $u,v\in [n]$. Then for all $w\in[n]$,
$$
   d_{t+1} < |\dist uw| = |\dist vw| \leq d_t,
$$
and thus, by \Cref{def:uniform embedding}, $a_{u,w} = a_{v,w} = t$.
This contradicts to the assumption that $A$ contains no repeating rows. 

We now prove that $\Pi$ us is monotone. Assume that $\Pi(1)<\Pi(2)$; we will show that $\Pi$ is increasing. 
We proceed an inductive proof such that $\Pi$ restricted to $[v]$ is strictly increasing for each $v$ from $2$ to $n$. 
The base case, $\Pi(1)<\Pi(2)$, holds by assumption.

Inductively, suppose $v\geq 2$ and $\Pi$ is strictly increasing restricted to $[v-1]$.
By the inductive hypothesis, it suffices to prove that $\Pi(v-1)<\Pi(v)$ to show $\Pi$ is strictly increasing restricted to $[v]$. 
Since there are no repeating vertices in $A$, let $w\in[n]$ be a vertex such that $a_{v-1, w}\neq a_{v,w}$.
\begin{enumerate}
    \item Suppose $w<v$. 
    By the definition of Robinson matrix, (\ref{robinson:intro}), $a_{v-1, w}>a_{v,w}$. Denote ${t_1} = a_{v-1,w}, {t_2} = a_{v,w}$ for some $t_1>t_2$. By \Cref{def:uniform embedding}, $d_{t_1}<d_{t_2}$ and
\begin{equation}\label{eqn:lemma:monotone}
\dist w{v-1} \leq d_{t_1} \leq d_{t_2+1}<\dist wv.
\end{equation} 
Rewrite (\ref{eqn:lemma:monotone}) to obtain $\Pi(v-1)<\Pi(v)$.
\item Suppose $w$ has $v<w$. 
First need to show $\Pi(v)\leq \Pi(w)$. 
Let ${t_1} = a_{v-1,w}, {t_2} = a_{v,w}$, where $t_1 < t_2$, by the definition of Robinson matrix, (\ref{robinson:intro}). 
Then, by \Cref{def:uniform embedding}, $d_{t_1}> d_{t_2}$ and 
\begin{equation}
        d_{t_2+1} < \dist{v}{w}\leq d_{t_2} \leq  d_{t_1+1} <\dist{v-1}{w}<d_{t_1}.
    \end{equation}
    Rewrite to obtain $\Pi(v-1)<\Pi(v)$.
\end{enumerate}
This concludes $\Pi(v-1)<\Pi(v)$. Thus, if $\Pi(1)<\Pi(2)$, then $\Pi$ is strictly increasing defined on any $[v]$ for $v$ from 2 to $n$.
If we assume $\Pi(1)>\Pi(2)$, then we can prove $\Pi$ is strictly decreasing with the same logic. 
Thus, $\Pi$ is strictly monotone defined on $[n]$.
\end{proof}

\begin{lemma}\label{lemma:strictly within}
Let $\Pi$ be a uniform embedding of $A\in\calS^n[k]$ with respect to $\bmd\in\bmD$ and suppose $\Pi$ is strictly monotone. 
Then, there exists an increasing $\Pi'$ with respect to $\bmd$ where all inequalities are strict, \ie, which satisfies Condition (\ref{eqn:shortcut2}).
\end{lemma}

\begin{proof}
Let $\Pi$ be a uniform embedding of $A\in\calS^n[k]$ with respect to $\bmd\in\bmD$ that is strictly monotone. We may assume that $\Pi$ is strictly increasing; if $\Pi$ is strictly decreasing, then $-\Pi$ is strictly increasing, and it can be easily checked that $-\Pi$ is also a uniform embedding of $A$ with respect to $\bmd$
Write $\bmd=(d_i)$. If $d_{t+1}<\dist uv<d_t$ where $t = a_{u,v}$ for all $u,v\in[n]$ with $u<v$, then $\Pi' = \Pi$ satisfies the statement of the lemma. Therefore, we assume there exists at least one pair $u,v\in[n]$ such that $\dist uv = d_t$, where $t = a_{u,v}$. 
Let $u\in[n]$ be then minimum vertex (index) such that there is $v\in[n]$ with $\Pi(v)-\Pi(u) = d_t$ where $t = a_{u,v}$.
Let $v$ be the minimum vertex with this property. 
Notice, for any $i,j\in[n]$ with $i<j$, $t = a_{i,j}$ implies that $d_{t+1} < \dist ij$, and thus $\Pi(j)>d_{t+1} + \Pi(i)$. Let $\epsilon = \min_{i,j\in[n],i<j} \{\dist ij-d_{t+1}:a_{i,j} = t\}$, and notice that $\epsilon > 0$. Then, define 
$$\Pi_0(i) = \left\{
\begin{array}{ll}
    \Pi(i) &\text{for }i\leq u,  \\
    \Pi(i)-\frac\epsilon2 & \text{for }i> u.
\end{array}
\right.$$
Such $\Pi_0$ is a uniform embedding of $A$ with respect to $\bmd$:
\begin{enumerate}
\item For any pairs $i,j$ both in $\{1,\dots, u\}$ or both in $\{u+1,\dots, n\}$, $\dist ij = \Pi_0(j)-\Pi_0(i)$; 
\item For any $i\in\{1,\dots, u\}, j\in\{u+1,\dots, n\}$, 
$$\dist ij - \epsilon < \Pi(j) - \frac\epsilon2 - \Pi(i) = \Pi_0(j)-\Pi_0(i),$$ where the inequality holds since $\epsilon>0$ and the equality is given by the definition of $\Pi_0$. 
\end{enumerate}
Observe that $\Pi_0(v)-\Pi_0(u) = \Pi(v) - \frac\epsilon2-\Pi(u) = d_t-\frac\epsilon2 <d_t$, so that pair $u,v$ satisfies (\ref{eqn:shortcut2});
and there is no new pairs $u', v'\in[n]$ so that $\Pi_0(v')-\Pi_0(u')= d_t.$  
Iteratively, obtain $\Pi_1, \Pi_2,\dots, \Pi_r$ until all such $u,v$ pairs, where $t= a_{u,v}$ and $\dist uv=d_t$, are adjusted to satisfy (\ref{eqn:shortcut2}). 
Then, $\Pi' = \Pi_r$ is a uniform embedding of $A$ with respect to $\bmd$ with all pairs of $u<v$, $a_{u,v} = t\iff d_{t+1}<\dist uv<d_t$ (\ie, (\ref{eqn:shortcut2})).
\end{proof}

\begin{proofof}{\Cref{thm:monotone}}
Suppose $A$ has a uniform embedding $\Pi'$ with respect to $\bmd = (d_i)$. Let $I$ be an index set that contains maximum number of non-repeating rows by their first appearance: for all $j\not\in I$, there is $i\in I$ such that $i,j$ are repeating rows and $i<j$.  
Consider $A[I]$, the submatrix of $A$ defined by index set $I$. That is,  $A[I]$ is obtained from $A$ by removing all repeating rows after the first appearance. 
 let $\Pi'' = \Pi'|_{I}$, the uniform embedding $\Pi'$ restricted to $v\in I$. It is easy to verify that $\Pi''$ is a uniform embedding of $A[I]$ with respect to $\bmd$.
By definition, the induced submatrix $A[I]$ is a Robinson matrix that contains no repeating rows and has a uniform embedding $\Pi''$. Therefore $\Pi''$ is strictly monotone by \Cref{lemma: no repeating}. By 
 \Cref{lemma:strictly within}, this implies that there exists $\Pi_0:I\to\rr$ that is a strictly increasing uniform embedding of $A[I]$ with respect to $\bmd$ that satisfies Condition (\ref{eqn:shortcut2}) (\ie, all inequalities are strict).
For each $i\in I$, let 
\begin{equation}\label{eqn:thm:monotone}
\begin{split}
    2\epsilon_i = \min & 
\{d_{t}- (\Pi_0(i)-\Pi_0(j)):j<i, j\in I, a_{i,j}=t\} \cup \\ 
&\{(\Pi_0(j)-\Pi_0(i))-d_{t+1}: i<j, j\in I,  a_{i,j} = t \}.
\end{split}
\end{equation}
Notice, $\Pi_0$ is a uniform embedding of $A[I]$ that is strictly increasing, \ie, for all $i,j\in I$, if $i<j$, then $d_{t+1} < \Pi_0(j)-\Pi_0(i)<d_t$, 
and if $j<i$, then $d_{t+1}<\Pi_0(i)-\Pi_0(j) < d_t$: therefore $2\epsilon>0$ (and thus $\epsilon>0$). 


We now extend $\Pi_0$ to a uniform embedding of $A$.
Let $i<j$ be a consecutive pair in $I$ (\ie, $i,j\in I$, there is no $k\in I$ such that $i<k<j$). Let $r_i$ be the size of index set $\{i+1, \dots, j-1\}$ where rows $i, i+1, \dots, j-1$ are repeating rows in $A$ that are identical to row $i$. 
Define $\Pi:[n]\to\rr$ where $\Pi|_{I} = \Pi_0$, and for each $i\in I$, define $\Pi$ for $j\not\in I$ by 
$$\Pi(i+k) = \Pi(i)+\frac{k}{r_i} \epsilon_i \quad\text{for }1\leq k\leq r_i.
$$
Do this for every $i\in I$ to complete the definition of $\Pi$. We verify that $\Pi$ is a uniform embedding of $A.$ 
We divide into two cases, such that for all $j\in I, j\neq i,$ either $i<j$ or $j<i$, we verify for all $k\in [r_i]$, $\dist{i+k}j$ or $\dist j{i+k}$ satisfies (\ref{eqn:shortcut2}).
\begin{enumerate}
\item First, assume $j<i$. Let $1\leq k \leq r_i$. By definition, $\Pi(i+k) - \Pi(j) = \Pi(i)-\Pi(j)+ \frac{k}{r_i}\epsilon_i$. Notice the following inequalities,
\begin{equation*}
\begin{split}
    \Pi(i)-\Pi(j) + \frac{k}{r_i}\epsilon_i  &<
\Pi(i)-\Pi(j) + \epsilon_i \\  
& \leq 
\Pi(i)-\Pi(j) + d_t - (\dist ji)\\
& = d_t
\end{split}    
\end{equation*}
where $t = a_{i,j}$. This gives $\dist j{i+k} < d_t$. Recall that $a_{i+k, j} = t = a_{i,j}$ since $i, i+k$ are repeating rows, this gives:
$$ a_{i+k, j} = a_{i,j} = t \iff
d_{t+1}<\dist ji < \dist j{i+k} < d_t. $$
\item Next, assume $i<j$ and $1\leq k \leq r_i$. By definition, 
$ \dist{i+k}j = 
\Pi(j)- (\Pi(i) + \frac k{r_i}\epsilon_i) $ 
and observe that
\begin{equation*}
\begin{split}
\Pi(j)- \Pi(i) - \frac k{r_i}\epsilon_i  & > \dist ij -\epsilon_i\\
 & \geq  \dist ij - (\dist ij - d_{t+1})\\
 & = d_{t+1}
\end{split}
\end{equation*}
where $t = a_{i,j}$. This gives $ \dist ij >d_{t+1}$. Recall $a_{i+k,j} = a_{i,j} = t$ and 
$$ a_{i+k,j} = t = a_{i,j}\iff
d_{t+1}<\dist ij < \dist {i+k}j < d_t.
$$
\end{enumerate}
So the two cases establish that $\Pi$ is a uniform embedding of $A$ with respect to $\bmd$ and satisfies (\ref{eqn:shortcut2}).
\end{proofof}

\subsection{Bounds, Walks, and Their Concatenation}

We give the proof of \Cref{lemma: concatenation} in this subsection; before proving it, we need a supplementary lemma on decomposing a walk. 
\begin{lemma}\label{remark: decompose cycle}
Let $W=\path{w_0,\dots, w_p}$ be a walk, and let $W_1 = \path{w_0, \dots, w_s}$ and $W_2 = \path{w_s, \dots, w_p}$ be a decomposition of $W$ into two walks. 
If $W$ is an upper-bound-walk, then both $W_1$ and $W_2$ are upper-bound-walks. 
Further, $\beta^+(W) = \beta^+(W_1)+\beta^+(W_2)$. 
\end{lemma}
\begin{proof}
Let $W=\path{w_0,\dots, w_p}$ be an upper-bound-walk, and let $W_1 = \path{w_0, \dots, w_s}$ and $W_2 = \path{w_s, \dots, w_p}$ breaks $W$ into two walks, $W=W_1 + W_2$. Note that $\edge{w_{i-1}}{w_i}$ is an edge or a \nedge{} follows \Cref{def:bound-walk} for all $i\in[p]$.
Then, $W_1$ is an upper-bound-walk since $\edge{w_{i-1}}{w_i}$ is an edge or a \nedge{} satisfies \Cref{def:bound-walk} for all $i\in[s]$; $W_2$ is an upper-bound-walk since $\edge{w_{i-1}}{w_i}$ is an edge or a \nedge{} satisfies \Cref{def:bound-walk} for all $i\in\{s+1, \dots, p\}$.
Further, by \Cref{def:beta vu}, we have that 
\begin{equation}\begin{split}
       \beta^+(W) 
&= \sum^{p}_{i=1} \beta^+(w_{i-1}, w_i) \\
&= \sum^{s}_{i=1} \beta^+(w_{i-1}, w_i) + \sum^{p}_{i=s+1} \beta^+(w_{i-1}, w_i) \\
&= \beta^+(W_1)+\beta^+(W_2).  
\end{split}\end{equation}
\end{proof}


We restate {\Cref{lemma: concatenation} here:\\

\noindent{\bf \Cref{lemma: concatenation}. }{\em 
Let $A\in \mathcal{S}^n[k]$.
For any $u,v\in [n]$ and any $(u,v)$-walk $W$, if $W$ is an upper-bound-walk then $\beta^+(W)$ is an upper bound on $(u,v)$, and if $W$ is a $(u,v)$-lower-bound-walk then $\beta^-(W)$ is a lower bound on $(u,v)$.
}

\begin{proofof}{\Cref{lemma: concatenation}}
We give a proof by induction on the length $p$ of $W$. 
When $p=1$ , then $W =\path{u,v}$. Then $\beta^+(W) = \beta^+(u,v)$, which  is an upper bound on $(u,v)$ by \Cref{def:bound} and Equation \ref{def:beta vu}. 

Suppose then that the statement of the lemma holds for all $q<p$.
Suppose that $W = \path{w_0,\dots, w_{p-1},w_p}$ is a $(w_0, w_p)$-upper-bound-walk of length $p$. 
By \Cref{remark: decompose cycle}, we have that $W'$ and $\path{w_{p-1}, w_p}$ are upper-bound-walks of length $p-1$ and $1$, respectively.
Then, by inductive hypothesis, $\beta^+(W')$ and $\beta^+(w_{p-1}, w_p)$ are upper bounds on $(w_0, w_{p-1})$ and $(w_{p-1}, w_p)$. 
Therefore, the following inequalities are implied by Inequality system (\ref{eqn:shortcut2}):
\begin{eqnarray*}
     \dist{w_{0}}{w_{p-1}} & < &\beta^+(W')^\top\bmd\text{, and } \\
     \dist{w_{p-1}}{w_p} & <  &\beta^+(w_{p-1}, w_p)^\top\bmd.
\end{eqnarray*}
Combine the two inequalities, we have that the following is also implied by (\ref{eqn:shortcut2}):
\begin{equation}\begin{split}
\dist{w_0}{w_p} &= (\dist{w_{0}}{w_{p-1}}) + (\dist{w_{p-1}}{w_p}) \\
    & < \beta^+(W')^\top\bmd+\beta^+(w_{p-1}, w_p)^\top\bmd \\
    & = \beta^+(W)^\top\bmd.
    \end{split}
\end{equation}
This implies that $\beta^+(W)^\top\bmd$ is an upper bound on $\dist{w_0}{w_p}$. This concludes the induction step.

With a similar argument, we have that $\beta^-(W)$ is a lower bound on $(w_0, w_p)$ if $W$ is a $(w_0, w_p)$-lower-bound-walk. 
\end{proofof}

\subsection{From walks to paths}
We will prove \Cref{lemma:loopless} in this section. This is the lemma that states 
that for any bound derived from a walk, there is an equal or tighter bound derived from a path.


\begin{proofof}{\Cref{lemma:loopless}}
 Let $A\in\calS^n[k]$  be a \diagincmat{} and let $\bmd\in\bmD$, and suppose $\bmd$ satisfies (\ref{eqn:condition_d}). Suppose $W$ is a $(u,v)$-upper-bound-walk $W$.
We will show that  there exists a $(u,v)$-upper-bound-path $W'$ such that \begin{equation}
    (\beta^+(W'))^\top\bmd \leq \beta^+(W)^\top \bmd.
\end{equation}
The statement is trivial when $W$ is already a $(u,v)$-upper-bound-path. 
Suppose then that $W=\path{u=w_0,e_1,w_1,\dots, w_p=v}$ contains a cycle. Precisely,
assume that $w_i=w_{i+l}$.
Decompose the walk $W$ into 
$$\begin{array}{rl}
W_1 &= \path{u=w_0,e_1,\dots, w_i}, \\
C &=\path{w_i, e_{i+1},\dots, w_{i+l}}, \\
W_2 &= \path{w_{i+l}, e_{i+l+1},\dots, w_p=v}.
\end{array}$$
Denote $W' = W_1 + W_2$. 
As in \Cref{remark: decompose cycle}, $W_1$, $C$, and $W_2$ are all upper-bound-walks, and $\beta^+(W) = \beta^+(W') + \beta^+(C)$. 
By \Cref{lemma:ub cycle geq 0}, for any $\bmd\in\bmD$ satisfying (\ref{eqn:condition_d}) we have that $ \beta^+(C)^\top\bmd \geq 0$, and thus
$$\beta^+(W)^\top\bmd  =  \beta^+(W')^\top\bmd + \beta^+(C)^\top\bmd \geq  \beta^+(W')^\top\bmd.$$
 Thus, we can iteratively remove any cycle in the walk $W$, finding walks with equal or tighter bounds each time, until obtaining a path. 
 A similar argument proves the statement for lower-bound-walks

\end{proofof}
\subsection{A Partial Order on Bounds}

We will devote this section to prove \Cref{thm:bound_equivalence}, which we restate here.

\noindent{\bf \Cref{thm:bound_equivalence}. }{\em 
Let $\bma, \bmb\in\zz^k$, then
$\bma\preceq \bmb \iff \bma^\top\bmd\leq \bmb^\top\bmd\text{ for all }\bmd\in\bmD.$
}

We need some supplementary definitions and lemmas to prove this theorem. 
The intuition of proving this theorem is to construct a ``buffer" bound $\bmc\in\zz^k$ so that $\bma^\top\bmd\leq\bmc^\top\bmd\leq\bmb^\top\bmd$ holds for any $\bmd\in\bmD$. 
The construction is technical; therefore, we first give an example. 
\begin{example}\label{example: rearrange weights}
Consider $\heartsuit, \diamondsuit, \clubsuit, \spadesuit$ to be four types of objects with weights $d_1, d_2, d_3, d_4$ where $d_1>d_2>d_3>d_4>0$. Then, a collection of $a_1$ number of $\heartsuit$, $a_2$ number of $\diamondsuit$, $a_3$ number of $\clubsuit$, and $a_4$ number of $\spadesuit$ together has weight $\bma^\top\bmd$, where $\bma = (a_1, a_2, a_3, a_4)^\top$ and $\bmd = (d_1,d_2,d_3,d_4)^\top$. We consider two collections of such objects, and arrange them with into four slots: 
\begin{center}
    \begin{tabular}{c|c|c|c|c|c}
    Collection 1 & $\heartsuit\heartsuit\heartsuit\heartsuit$ & $\diamondsuit\diamondsuit$ & $\clubsuit\clubsuit\clubsuit$ & $\spadesuit\spadesuit$ & \\
    Weights & $4d_1$ & $2d_2$ & $3d_3$ & $2d_4$ & \\\hline
    Collection 2 & $\heartsuit\heartsuit\heartsuit$ & $\diamondsuit$ & $\clubsuit$ & $\spadesuit\spadesuit\spadesuit\spadesuit\spadesuit$ & \\
    Weights & $3d_1$ & $d_2$ & $d_3$ & $5d_4$ & \\
\end{tabular}
\end{center}
Then, we rearrange collection 2 as the following
\begin{center}
\begin{tabular}{c|c|c|c|c|c}
    Collection 1 & $\heartsuit\heartsuit\heartsuit\heartsuit$ & $\diamondsuit\diamondsuit$ & $\clubsuit\clubsuit\clubsuit$ & $\spadesuit\spadesuit$ & \\
    Weights & $4d_1$ & $2d_2$ & $3d_3$ & $2d_4$ & \\\hline
    Collection 2 rearranged & $\heartsuit\heartsuit\heartsuit\spadesuit$ & $\diamondsuit\spadesuit$ & $\clubsuit\spadesuit$ & $\spadesuit\spadesuit$ & \\
    Weights & $3d_1 + d_4$ & $d_2 + d_4$ & $d_3+d_4$ & $2d_4$ & \\
\end{tabular}
\end{center}
Notice that the weight in each slot in collection 1 is greater than collection 2, since $d_i> d_4$ for any $i<3$. 
We construct collection 3 from collection 2 so that, in each slot,  $\spadesuit$ is replaced by another type: 
\begin{center}
\begin{tabular}{c|c|c|c|c|c}
    Collection 1 & $\heartsuit\heartsuit\heartsuit\heartsuit$ & $\diamondsuit\diamondsuit$ & $\clubsuit\clubsuit\clubsuit$ & $\spadesuit\spadesuit$ & \\
    Weights & $4d_1$ & $2d_2$ & $3d_3$ & $2d_4$ & \\\hline
    Collection 3 & $\heartsuit\heartsuit\heartsuit\heartsuit$ & $\diamondsuit\diamondsuit$ & $\clubsuit\clubsuit$ & $\spadesuit\spadesuit$ & \\
    Weights & $4d_1$ & $2d_2$ & $2d_3$ & $2d_4$ & \\\hline
    Collection 2 rearranged & $\heartsuit\heartsuit\heartsuit\spadesuit$ & $\diamondsuit\spadesuit$ & $\clubsuit\spadesuit$ & $\spadesuit\spadesuit$ & \\
    Weights & $3d_1 + d_4$ & $d_2 + d_4$ & $d_3+d_4$ & $2d_4$ & \\
\end{tabular}
\end{center}
Notice that the weight of each slot in collection 1 is higher (heavier) than or equal to the corresponding slot in collection 3. 
Therefore the total weight of collection 1 is heavier than collection 3. 
Also, notice that we constructed collection 3 from collection 2 by replacing $\spadesuit$ by something heavier, \ie, $\heartsuit, \diamondsuit$, or $\clubsuit$; therefore, the total weight of collection 3 is heavier than collection 2. 
Let $\bmb = (b_i), \bma = (a_i), \bmc = (c_i)\in\zz^4$ denotes the number of $\heartsuit,\diamondsuit,\clubsuit,\spadesuit$ in each collection 1, 2, 3; then this weight comparison can be expressed as $\bma^\top\bmd\leq\bmc^\top\bmd\leq\bmb^\top\bmd.$
\end{example}

The above example demonstrates the technique of the proof of \Cref{thm:bound_equivalence}.
Suppose two vectors $\bma = (a_i), \bmb=(b_i)\in\zz^k$ have that $\bma\preceq\bmb$, then $\bma^\top\bmd\leq\bmb^\top\bmd$ is obvious if $a_i\leq b_i$ for all $i\in[k]$. 
If the two vectors cannot be compared component-wise (\ie, $a_i\leq b_i$ for all $i\in[k]$), then we rearrange the components and construct a ``buffer" vector $\bmc=(c_i)$ (such as collection 3), so that
\begin{itemize}
    \item $c_i\leq b_i$ for all $i\leq k$ and
    \item we may obtain $\bma^\top\bmd\leq \bmc^\top\bmd$ easily, according to the construction.
\end{itemize}


We now define how to construct the vector $c$.
\begin{definition}\label{def: buffer vector c}
Let $\bma = (a_i),\bmb = (b_i)\in\zz^k$. 
Suppose $\bma\preceq\bmb$, define $\bmc$ with the following steps.

Iteratively, for $t=1, \dots, k$, define $\{c_{t,i}\}_{i\in[t]}$ as follows. 
First, fix $t\in [k]$ and define auxiliary sequences $\{c_{t,i}\}_{i\in[t]}$ $\{e_{t,i}\}_{i\in[t-1]}$ and $\{f_{t,i}\}_{i\in[t]}$, as follows. For all 
Define $e_t = [a_t - b_t]_+$ and $f_{t,1} = e_t$. 
For $t\geq 2$ and $i = 1,\dots, t-1$, let
\begin{eqnarray}\label{eqn: recursive def ef}
    e_{t, i} &   =&\min\{f_{t,i}, b_i - c_{t-1,i}\},  \\
    f_{t, i+1} & =&  f_{t,i} - e_{t,i},\notag\\
    c_{t,i} &=& c_{t-1, i} + e_{t,i}\notag
\end{eqnarray}
Then, define $c_{t,t} = a_t - e_t$.
Finally, define $\bmc = (c_{k,i})_{i\in[k]}$. 
\end{definition}


\begin{lemma}\label{lemma: non-neg seq}
Let $\bma =(a_i), \bmb = (b_i)\in\zz^k$ and suppose $\bma\preceq\bmb$. 
Following the notations in \Cref{def: buffer vector c}, for all $t\in[k]$, the following holds:
\begin{enumerate}
    \item Sequence $\{f_{t,i}\}_{i\in[t]}$ is a non-negative and decreasing sequence. 
    \item For all $i\in[t]$, $c_{t,i}\leq b_i$.
\end{enumerate}
\end{lemma}
\begin{proof}
Let $\bma =(a_i), \bmb = (b_i)\in\zz^k$ and suppose $\bma\preceq\bmb$. Following the notations in \Cref{def: buffer vector c}.
We will give a proof by induction on $t$.
Consider the base case $t=1$. 
The assumption that $\bma\preceq\bmb$ gives that $a_1 \leq b_1.$ 
Therefore, $e_t = [a_1 - b_1]_+ = 0$ and $c_{1,1} = a_1\leq b_1$. 
This show part 1 of the statement of the lemma. Since $f_{t,1}=a_1\geq 0$ the sequence with one element, $\{f_{1,1}\}$, is non-negative and decreasing, which shows part 2.

Fix $t\geq 2$, and assume parts 1 and 2 of the statement hold for $t-1$.
Note that we defined $e_t = [a_t - b_t]_+\geq 0$. So $f_{t,1} = e_t \geq 0$
We will prove that $\{f_{t,i}\}_{i\in[t]}$ is decreasing by induction on $i$; this then shows that the sequence is non-negative.
The base case, $i=1$, is trivial.

Inductively, for all $2\leq i\leq t$, we have that $f_{t,i}\geq f_{t,i-1}\geq 0$. We will show that $f_{t,i+1}\geq f_{t,i}\geq 0$. 
By \Cref{def: buffer vector c}, $e_{t,i} = \min\{f_{t,i}, b_i - c_{t-1, i}\}$. 
By part 2 of the induction hypothesis (of the induction on $t$)   we have that $b_i \geq c_{t-1,i}$. Therefore, $b_i - c_{t-1,i}\geq 0$. 
Then, since $f_{t,i}$ and $b_i - c_{t-1,i}$ are both non-negative, we have that $e_{t,i}$ is non-negative by definition. 
Thus, we have that $f_{t,i+1} = f_{t,i} + e_{t,i} \geq f_{t,i}\geq 0$. This concludes the proof that  $\{f_{t,i}\}_{i\in[t]}$ is decreasing, and thus, part 2 holds for $t$.

Finally, we need to show that $c_{t,i}\leq b_i$. 
Since $\{f_{t,i}\}$ is a decreasing sequence, as the above inductive proof shows, we have that $e_{t,i}\geq 0$ for all $i\in[t-1]$. 
Then, by \Cref{def: buffer vector c}, for $i<t$, we have that $0\leq e_{t,i}\leq b_i - c_{t-1, i}$; 
or equivalently, $c_{t-1, i}\leq b_i$. 
\Cref{def: buffer vector c} defines $c_{t,t} = a_t - e_t = a_t - [a_t - b_t]_+$: if $a_t\leq b_t$, then $[a_t- b_t]_+= 0$ and we have that $c_{t,t} = a_t \leq b_t$; if $a_t> b_t$, then $[a_t - b_t]_+=a_t-b_t$ and we have that $c_{t,t} = a_t - a_t + b_t = b_t$.
Therefore, we conclude that $c_{t,i}\leq b_i$ for all $i\leq t.$
\end{proof}
\begin{lemma}\label{cor: c<=b}
Let $\bma =(a_i), \bmb = (b_i)\in\zz^k$ and suppose $\bma\preceq\bmb$. 
Following the notations in \Cref{def: buffer vector c}, for all $t\in[k]$, the following holds:
\begin{itemize}
    \item[(i)] $e_t = \sum_{i=1}^{t-1} e_{t,i}$. 
    \item[(ii)] $\sum_{i=1}^{t} a_i = \sum_{i=1}^{t} c_{t,i}$.
\end{itemize}
\end{lemma}
\begin{proof}
Let $\bma =(a_i), \bmb = (b_i)\in\zz^k$ and suppose $\bma\preceq\bmb$. 
We give a proof by induction on $t$.

For the base case, when $t=1$, \Cref{def: buffer vector c} defines $c_{1,1} = a_{1} + e_1$. 
And since $\bma\preceq\bmb$ implies that $a_1\leq b_1$, so $e_1 = [a_1 - b_1]_+ = 0$, so part (i) holds.   
Therefore, $a_1 = c_{1,1}$, so part (ii) holds.

Inductively, suppose that, for  $t>1$, we have $e_{t-1} = \sum_{i=1}^{t-2} e_{t-1,i}$ and $\sum_{i=1}^{t-1} a_i = \sum_{i=1}^{t-1} c_{t-1,i}$. 
We divide into two cases: when $a_t\leq b_t$ and when $a_t>b_t$.
If $a_t\leq b_t$, then $e_t = [a_t - b_t]_+ =0$ and $f_{t,1} = e_t$. 
From \Cref{lemma: non-neg seq}, we have that sequence $\{f_{t,1}\}$ is decreasing and $b_i - c_{t,i}\geq 0$;
therefore, when $a_t\leq b_t$, $f_{t,i}= 0$ for all $i\in[t]$ and $e_{t,i} = 0$ for all $i<t$. 
Thus, $e_t = \sum_{i=1}^{t-1}  e_{t,i} = 0$. 

Next, suppose that $a_t > b_t$. By \Cref{def: buffer vector c}, $e_t = a_t - b_t,$ so $c_{t,t} = b_t$. 
Notice that \Cref{eqn: recursive def ef} defines $f_{t,i+1} = f_{t,i}-e_{t,i}$, and rewriting the equation we obtain $f_{t,i} = f_{t,i+1}+e_{t,i}$. 
Expand $e_{t} = f_{t,1}$ according to \Cref{eqn: recursive def ef}: 
\begin{equation}\begin{split}
    e_t = f_{t,1}  & = f_{t,2} + e_{t,1} \\
    & = f_{t, 3} + e_{t,2} + e_{t,1} \\
    & = f_{t, 4} + e_{t,3} + e_{t,2} + e_{t,1} \\
    & = \dots\\
    & = f_{t,t} + e_{t, t-1} + \dots + e_{t,1}.
\end{split}\end{equation}
Then, we need to show $f_{t,t} =0 $ so that $e_t$ can be written by the sum of $e_{t,i}$ only. 

From \Cref{lemma: non-neg seq}, we have that $f_{t,t}\geq 0$ since sequence $\{f_{t,i}\}_{i\in[t]}$ is non-negative. 
Suppose that $f_{t,t}>0$, then we have $e_{t,i} = b_i - c_{t-1, i}$ for all $i\in[t-1]$. 
That is, we know that $ b_i - c_{t-1,i}<f_{t,i}$ for any $i\leq t$.
This is true since, otherwise, if $e_{t,i}$ is defined by $f_{t,i}\leq b_i - c_{t-1,i}$ for some $i<t$, then $f_{t,i+1}= f_{t,i} - e_{t,i} = 0$;
then, $f_{t,t} = 0$ since we show that $\{f_{t,i}\}$ is a non-negative and decreasing sequence in \Cref{lemma: non-neg seq}. Then we obtain the following inequality.
\begin{equation}\label{eqn: ftt}
a_t - b_t = e_t 
= f_{t,t} + \sum_{i=1}^{t-1} e_{t,i} 
> \sum_{i=1}^{t-1} e_{t,i} 
= \sum_{i=1}^{t-1} (b_i-c_{t-1, i}).
\end{equation}
By the inductive hypothesis, we have $\sum_{i=1}^{t-1} a_i = \sum_{i=1}^{t-1} c_{t-1,i}$. Then, substitute and we have that 
$$
a_t - b_t > \sum_{i=1}^{t-1} (b_i-c_{t-1, i})= \sum_{i=1}^{t-1} b_i-\sum_{i=1}^{t-1} c_{t-1, i} = \sum_{i=1}^{t-1} b_i-\sum_{i=1}^{t-1} a_i. 
$$
And move terms in the above equation, we have that
$$
\sum_{i=1}^{t} a_i > \sum_{i=1}^{t} b_i
$$
However, we assumed that $\bma\preceq\bmb$ and thus $\sum_{i=1}^ta_i\leq\sum_{i=1}^tb_i$. This is a contradiction. 
Therefore, $f_{t,t} = 0$, and $e_t = \sum_{i=1}^{t-1} e_{t,i}.$

Finally, we will show $\sum_{i=1}^t a_i = \sum_{i=1}^t c_{t,i}$. 
Note that we have $ e_t = [a_t - b_t]_+ = \sum_{i=1}^{t-1} e_{t,i}$, $c_{t,t} = a_t - e_t$, and $c_{t,i} = c_{t,i-1} + e_i$.
If $a_t\leq b_t$, then $e_t = 0$ and $e_{t,i} = 0$ for all $i<t$, so we have $c_{t,t} = a_t$ and $c_{t,i} = c_{t-1,i}$. Combining with the inductive hypothesis, we have that 
$$\sum_{i=1}^{t} a_i = a_t + \sum_{i=1}^{t-1} a_i = c_{t,t} + \sum_{i=1}^{t-1}c_{t,i} = \sum_{i=1}^{t} c_{t,i}.$$
Now we suppose that $a_t > b_t$, then $e_t = a_t - b_t$ and $c_{t,t} = b_t$.
Thus, we also have the following equation:
\begin{equation}\begin{split}
\sum_{i=1}^t c_{t,i} 
&= c_{t,t} + \sum_{i=1}^{t-1} c_{t,i}  \\
&= b_{t} + \sum_{i=1}^{t-1} (c_{t-1,i} + e_{t,i}) \\
&= b_{t} + \sum_{i=1}^{t-1} c_{t-1,i} + \sum_{i=1}^{t}e_{t,i} \\
&= b_{t} + \sum_{i=1}^{t-1} a_{i} + e_t \\
&= b_{t} + \sum_{i=1}^{t-1} a_{i} + a_t-b_t\\
&= \sum_{i=1}^{t} a_{i},
\end{split}\end{equation}
which was what we want.
\end{proof}

\begin{proofof}{\Cref{thm:bound_equivalence}}
\thisway We first prove the forward direction. 
Suppose $\bm a, \bm b\in\zz^k$ such that $\bm a\preceq\bm b$. 
We will show $\bma^\top\bmd\leq \bmb^\top\bmd$ for all $\bmd\in\bmD$.

We construct $\bmc\in\zz^k$ use \Cref{def: buffer vector c} and we follow the notations in \Cref{def: buffer vector c}. 
We decompose the proof into two parts such that $\bma^\top\bmd\leq\bmc^\top\bmd$ and $\bmc^\top\bmd\leq\bmb^\top\bmd$, and thus the conclusion follows.

By \Cref{lemma: non-neg seq}, we have that $c_i\leq b_i$, and thus $c_id_i\leq b_i d_i$. Then 
$$\bmc^\top\bmd = \sum_{t=1}^{k}c_td_t\leq \sum_{t=1}^{k}b_td_t = \bmb^\top\bmd.$$
So $\bmc^\top\bmd \leq  \bmb^\top\bmd$ follows immediately.

To prove $\bma^\top\bmd\leq \bmc^\top\bmd$, we give an inductive proof such that, for all $t\in[k]$, $\sum_{i=1}^t a_{t,i}d_i \leq \sum_{i=1}^tc_{t,i}d_i$.
Let $e_t,$ $\{e_{t,i}\}_{i\in[t]}$, and $\{c_{t,i}\}_{i\in[t]}$ defined as in \Cref{def: buffer vector c}, that is, $e_t = [a_t - b_t]_+= \sum_{i=1}^{t-1} e_{t,i}$ and $c_{t,i} = c_{t-1, i} + e_{t, i}$. 

When $t=1$, we have $c_{1,1} = a_1$ by \Cref{cor: c<=b}. Then, $ a_1d_1  = c_{1,1}d_1 $. This is the base case of the inductive statement. 

Inductively, suppose that, for any $t>1$, we have $\sum_{i=1}^{t-1} a_i d_i\leq \sum_{i=1}^{t-1} c_{t-1,i} d_i$. 
We divide into two cases when $a_t\leq b_t$ and when $a_t>b_t.$
When $a_t\leq b_t$, \Cref{cor: c<=b} gives $c_{t,t} = a_t$ and $c_{t,i} = c_{t-1,i}$ for all $i<t$. 
Then, by the inductive hypothesis, $$\sum_{i=1}^{t} a_i d_i = a_td_t + \sum_{i=1}^{t-1} a_i d_i \leq c_{t,t} d_t +  \sum_{i=1}^{t-1} c_{t-1,i} d_i = c_{t,t} d_t +  \sum_{i=1}^{t-1} c_{t,i} d_i =\sum_{i=1}^{t} c_{t,i} d_i,$$ which satisfies the inductive hypothesis. 

Consider when $a_t > b_t,$ then $e_t = a_t - b_t>0$ and $c_{t,t} = b_t$. Consider the following sequence of inequalities. 
\begin{equation}
    \begin{split}
       \sum_{i=1}^t a_i d_i  & = a_t d_t + \sum_{i=1}^{t-1} a_i d_t\\
       & \leq a_t d_t  + \sum_{i=1}^{t-1} c_{t-1, i} d_t \quad\text{by the inductive hypothesis}\\ 
       & =  a_t d_t  + \sum_{i=1}^{t-1} c_{t-1, i} d_t + c_{t,t} d_t - c_{t,t} d_t \\ 
       & = (a_t- c_{t,t}) d_t  + \sum_{i=1}^{t-1} c_{t-1, i} d_t + c_{t,t} d_t \\
       & = e_t d_t + \sum_{i=1}^{t-1} c_{t-1, i} d_t + c_{t,t} d_t \\
           \end{split}
\end{equation}
We write $e_{t} = \sum_{i=1}^{t-1}e_{t,i-1}$ as in \Cref{cor: c<=b}; 
also note that $d_i>d_t$, for all $i<t$, implies that $e_{t,i} d_i \geq e_{t,i}d_t$ since $e_{t,i}\geq0$ by \Cref{lemma: non-neg seq}. 
Then, we have that 
\begin{equation}
    \begin{split}
      \sum_{i=1}^t a_i d_i & = (\sum_{i=1}^{t-1} e_{t,i}) d_t+ \sum_{i=1}^{t-1} c_{t-1, i} d_t + c_{t,t} d_t \\
       & \leq  (\sum_{i=1}^{t-1} e_{t,i} d_i) +  \sum_{i=1}^{t-1} c_{t-1, i} d_t + c_{t,t} d_t \\
       & = \sum_{i=1}^{t-1} (e_{t,i} + c_{t-1, i}) d_t + c_{t,t} d_t \\
       & = \sum_{i=1}^{t-1} c_{t, i} d_t + c_{t,t} d_t \\
       & = \sum_{i=1}^{t} c_{t, i} d_t 
    \end{split}
\end{equation}
as desired. Therefore, the inductive statement holds for case $t$. 

When $t = k$, by definition $\bmc = (c_i)_{i\in[k]}$ where $c_{i} = c_{k,i}$, rewrite $\sum_{i=1}^t a_i d_i \leq \sum_{i=1}^{t} c_{i} d_t $ as $\bma^\top\bmd\leq \bmc^\top\bmd,$ which was what we want.

Combine $\bma^\top\bmd\leq\bmc^\top\bmd$ and $\bmc^\top\bmd\leq\bmb^\top\bmd$, and we conclude that $\bma^\top\bmd\leq \bmb^\top\bmd$.

\thatway Now we prove the converse of the statement. We proceed with a proof by contrapositive. 
Let $\bma = (a_i), \bmb = (b_i)\in\zz^k$ and suppose $\bma\not\preceq\bmb$.
We will show there exists $\bmd\in\bmD$ such that 
$\bma^\top\bmd > \bmb^\top\bmd$.

Note that the statement $\bma\preceq\bmb$ is defined as, for all $t\in[k]$, we have that $\sum_{i=1}^ta_i\leq \sum_{i=1}^tb_i$; then, the negation of it, $\bma\not\preceq\bmb$, is that, exists $t\in[k],$ we have $\sum_{i=1}^ta_i>  \sum_{i=1}^tb_i$.
Suppose $t\in[k]$ is the minimal counterexample such that, for all $t'< t$, 
\begin{equation}\label{caseT_1}
    \sum_{i=0}^{t'}a_i\leq \sum_{i=0}^{t'}b_i
\end{equation}
and 
\begin{equation}\label{caseT}
    \sum_{i=1}^ta_i>\sum_{i=1}^tb_i.
\end{equation} 
Combining the two inequalities, we have
\begin{equation}\label{seq}
a_t -b_t>\sum_{i=1}^{t-1}b_i-\sum_{i=1}^{t-1}a_i\geq 0
\end{equation}

Divide both sides by $a_t-b_t$ (which is positive),
\begin{equation}
    1>\frac{\sum_{i=1}^{t-1}(b_i-a_i)  }{a_t-b_t}.
\end{equation}
Let $d_1 > d_t>0$ be so that 
\begin{equation}\label{equi:less than 1}
    1>\frac{\sum_{i=1}^{t-1}(b_i-a_i)  }{a_t-b_t}\cdot\frac{d_1}{d_t}> \frac{\sum_{i=1}^{t-1}(b_i-a_i)}{a_t-b_t}.
\end{equation}
\ie, pick a value for $d_t/d_1$ in the interval $( \sum_{i=1}^{t-1}(b_i-a_i)/(a_t-b_t), 1)$. 
Since, for all $\bmd\in\bmD$, $1 < i < t$ implies that $d_1 > d_i$, it follows that
\begin{equation}\label{Final_inequality}
\begin{split}
    \frac{\sum_{i=1}^{t-1}(b_i-a_i) }{a_t-b_t}\cdot\frac{d_1}{d_t}
     & = \sum_{i=1}^{t-1}\frac{(b_i-a_i)d_1}{(a_t-b_t)d_t}\\
     & \geq \sum_{i=1}^{t-1}\frac{(b_i-a_i)d_i}{(a_t-b_t)d_t}\\
\end{split}
\end{equation}

Since $a_t>b_t$ and $d_t>0$, multiplying both sides of \Cref{equi:less than 1} by $(a_t-b_t)d_t$ does not change the direction of the inequality. 
Combining \Cref{equi:less than 1} and (\ref{Final_inequality}), we have that 
$$(a_t-b_t)d_t > \sum_{i=1}^{t-1}(b_i-a_i)d_i.$$
It follows that $\sum_{i=1}^{t}a_id_i>\sum_{i=1}^{t}b_id_i$.
Let $\epsilon>0$ be so that $\sum_{i=1}^{t}a_id_i  = \sum_{i=1}^{t}b_id_i+\epsilon $. Then choose small enough $d_{t+1}>\dots> d_k$ so that $ \sum_{i=t+1}^k(b_i-a_i) d_i< \epsilon$ (\ie, use $d_i$ arbitrarily small so that $(b_i - a_i) d_i$ are small). 
Substitute it in the above equation so that $$\sum_{i=1}^{t}a_id_i >  \sum_{i=1}^{t}b_id_i + \sum_{i=t+1}^k(b_i-a_i) d_i$$ and so $ \sum_{i=1}^{k}a_id_i >  \sum_{i=1}^{k}b_id_i $ by moving terms, 
and thus $\bm a^\top \bm d> \bm b^\top \bm d$, which was what we want.
\end{proofof}

\end{document}